\documentclass[preprint]{elsarticle}

\usepackage{lineno,hyperref}
\usepackage{hyperref}
\usepackage{mathtools,amsmath,amssymb,amsthm}

\journal{Mediterranean Journal of Mathematics}

\usepackage{fullpage}

\newtheorem{result}{Deep Result}
\newtheorem{lema}{Lemma}
\newtheorem{theorem}{Theorem}
\newtheorem{coro}{Corollary}
\newproof{pf}{Proof}

\DeclarePairedDelimiter\norm{\lVert}{\rVert}

\makeatletter
\def\th@plain{%
  \thm@notefont{}
  \itshape 
}
\def\th@definition{%
  \thm@notefont{}
  \normalfont 
}
\makeatother
\DeclareFontFamily{U}{matha}{\hyphenchar\font45}
\DeclareFontShape{U}{matha}{m}{n}{ <5> <6> <7> <8> <9> <10> gen * matha <10.95> matha10 <12> <14.4> <17.28> <20.74> <24.88> matha12 }{}
\DeclareSymbolFont{matha}{U}{matha}{m}{n}
\DeclareFontFamily{U}{mathx}{\hyphenchar\font45}
\DeclareFontShape{U}{mathx}{m}{n}{ <5> <6> <7> <8> <9> <10> <10.95> <12> <14.4> <17.28> <20.74> <24.88> mathx10 }{}
\DeclareSymbolFont{mathx}{U}{mathx}{m}{n}
\DeclareMathSymbol{\obot} {2}{matha}{"6B}
\DeclareMathSymbol{\bigobot} {1}{mathx}{"CB}

\newcommand\mc{\mathcal}
\newcommand\mb{\mathbb}
\newcommand\mf{\mathbf}
\newcommand\ve{\varepsilon}

\newcommand\xt{\mathfrak{T}}

\DeclareMathOperator{\Span}{Span}
\DeclareMathOperator{\Ker}{Ker}

\DeclareMathOperator{\Ric}{Ric}

\DeclareMathOperator{\id}{Id}

\begin{document}

\begin{frontmatter}

\title{The Orthogonality Principle for Osserman Manifolds}


\author[matf]{Vladica Andreji\'c}
\ead{andrew@matf.bg.ac.rs}
\author[matf]{Katarina Luki\'c}
\ead{katarina.lukic@matf.bg.ac.rs}

\address[matf]{Faculty of Mathematics, University of Belgrade, Belgrade, Serbia}




\begin{abstract}
We introduce a new potential characterization of Osserman algebraic curvature tensors. 
An algebraic curvature tensor is Jacobi-orthogonal if $\mc{J}_XY\perp\mc{J}_YX$ holds for all $X\perp Y$,
where $\mc{J}$ denotes the Jacobi operator.
We prove that any Jacobi-orthogonal tensor is Osserman, while all known Osserman tensors are Jacobi-orthogonal.
\end{abstract}


\end{frontmatter}


\section{Introduction}

A connected Riemannian manifold is called two-point homogeneous if its isometry group is transitive on equidistant pairs of points.
These are the simplest and most beautiful Riemannian manifolds, with which we usually compare other spaces, and we call them the model spaces. 
The model spaces can be observed from plenty different points of view.
For example, it is well known (see Wolf \cite[Lemma 8.12.1]{Wo}) that a connected Riemannian manifold is isotropic (it has the geometry that does not depend on directions) if and only if it is two-point homogeneous.

Moreover, it is known (see Helgason \cite[p.535]{He}) that a model space is either flat or locally isometric to a rank one symmetric space. 
As a consequence of the classification, any model space is isometric to one of the following:
a Euclidean space; a sphere; a real, complex or quaternionic, projective or hyperbolic space;
or the Cayley projective or hyperbolic plane.
More precisely, the classification of these spaces includes:
$\mb{R}^n$, $\mf{S}^n$, $\mb{R}\mf{P}^n$, $\mb{C}\mf{P}^n$, $\mb{H}\mf{P}^n$, $\mb{O}\mf{P}^2$,
$\mb{R}\mf{H}^n$, $\mb{C}\mf{H}^n$, $\mb{H}\mf{H}^n$, and $\mb{O}\mf{H}^2$.
However, note that there are isomorphisms (isometries up to a homothety) in low dimensions:
$\mb{R}\mf{P}^1\cong \mf{S}^1$, $\mb{C}\mf{P}^1\cong \mf{S}^2$, $\mb{H}\mf{P}^1\cong \mf{S}^4$, $\mb{O}\mf{P}^1\cong \mf{S}^8$,
$\mb{C}\mf{H}^1\cong \mb{R}\mf{H}^2$, $\mb{H}\mf{H}^1\cong \mb{R}\mf{H}^4$, $\mb{O}\mf{H}^1\cong \mb{R}\mf{H}^8$.

Local isometries of a locally two-point homogeneous spaces act transitively on the sphere bundle of unit tangent vectors,
and therefore fix the characteristic polynomial of the Jacobi operator there.
Such Riemannian manifolds in which the characteristic polynomial (or equivalently, the eigenvalues and their multiplicities) 
of a Jacobi operator $\mc{J}_X$ is independent of $X$ from the unit tangent bundle are called (globally) Osserman manifolds.
One of the key questions in Riemannian geometry is whether the converse is true (every Osserman manifold is locally a model space),
and it is known as the Osserman conjecture (see Osserman \cite{O}). 
Let us remark that Nikolayevsky \cite{Ni4,Ni2,Ni1,Ni3} proved the conjecture in all cases, 
except the manifolds of dimension $16$ whose reduced Jacobi operator has an eigenvalue of multiplicity $7$ or $8$.

We are interested in algebraic properties related to algebraic curvature tensors that may be realized at a point of an Osserman manifold.
A well known such property is the Jacobi-duality, which has been shown to provide a characterization of Osserman tensors.
Recently, the first author introduced the Jacobi-proportionality, a property shared by all known Osserman tensors,
and provided some new insights into the Osserman conjecture \cite{A9}.  

Here we introduce a new feature that we call the Jacobi-orthogonality, where $\mc{J}_XY\perp\mc{J}_YX$ holds for all $X\perp Y$.
It turns out that tensors with such a property are Osserman (Corollary \ref{orto1c}), while it holds for all known Osserman tensors (Corollary \ref{orto23c}).

\section{Preliminaries}

Let $(\mc{V},g)$ be a (positive definite) scalar product space of dimension $n$ and
let $\ve_X=g(X,X)=\norm{X}^2$ be the squared norm of $X\in\mc{V}$.
A tensor $R\in\xt^0_4(\mc{V})$ is said to be an algebraic curvature tensor on $(\mc{V},g)$ 
if it satisfies the usual $\mb{Z}_2$ symmetries as well as the first Bianchi identity.
Raising the index we obtain an algebraic curvature operator $\mc{R}=R^\sharp\in\xt^1_3(\mc{V})$.
The Jacoby operator is a self-adjoint linear operator $\mc{J}_X\colon\mc{V}\to\mc{V}$ given by $\mc{J}_X Y=\mc{R}(Y,X)X$. 
Since $g(\mc{J}_XY,X)=0$ and $\mc{J}_XX=0$, the Jacobi operator $\mc{J}_X$ for any nonzero $X\in\mc{V}$ 
is completely determined by its restriction $\widetilde{\mc{J}}_X\colon X^\perp \to X^\perp$ called the reduced Jacobi operator.

Let $R$ be an algebraic curvature tensor on $(\mc{V},g)$.
We say that $R$ is Osserman if the characteristic polynomial $\widetilde{\omega}_X(\lambda)=\det(\lambda\id-\widetilde{\mc{J}}_X)$
of a reduced Jacobi operator $\widetilde{\mc{J}}_X$ is independent of unit $X\in\mc{V}$.
More generally, we say that $R$ is $k$-root if $\widetilde{\mc{J}}_X$ has exactly $k$ distinct eigenvalues for any nonzero $X\in\mc{V}$.

A Clifford family of rank $m$ is an anti-commutative family of skew-adjoint complex structures $J_i$ for $1\leq i\leq m$, 
which yields the Hurwitz relations, $J_i J_j+J_j J_i=-2\delta_{ij}\id$ for $1\leq i,j\leq m$.
We say that $R$ is Clifford if it has a form associated with a Clifford family,
\begin{equation}\label{clifford}
R=\mu_0 R^1+\sum_{i=1}^m \mu_i R^{J_i},
\end{equation}
for some $\mu_0,\dots,\mu_m\in\mb{R}$,
where any skew-adjoint endomorphism $J$ on $\mc{V}$ generates an algebraic curvature tensor $R^J$ defined by
\begin{equation*}
R^J(X,Y,Z,W) = g(JX,Z)g(JY,W) -g(JY,Z)g(JX,W) +2g(JX,Y)g(JZ,W),
\end{equation*}
for $X,Y,Z,W\in\mc{V}$, while $R^1$ is the algebraic curvature tensor of constant sectional curvature $1$.
It is well known that any Clifford $R$ is Osserman, while the converse is true in a huge number of cases,
which we see in the following deep result proved by Nikolayevsky \cite{Ni4,Ni2,Ni1,Ni3}.

\begin{result}\label{res2}
If an Osserman algebraic curvature tensor is not Clifford then it has dimension 16 such that 
the reduced Jacobi operator has an eigenvalue of multiplicity 7 or 8.
\end{result}

We are interested in various characterizations or generalizations of Osserman algebraic curvature tensors. 
One nice property of Osserman tensors is the so-called Raki\'c duality principle introduced by Raki\'c in 1999 \cite{Ra2}. 
According to the original definition we say that $R$ is Jacobi-dual if $\mc{J}_XY=\lambda Y$ 
for mutually orthogonal unit vectors $X,Y\in\mc{V}$ implies $\mc{J}_YX=\lambda X$.
In other words, the Jacobi-duality says that $Y$ is an eigenvector of $\mc{J}_X$ if and only if $X$ is an eigenvector of $\mc{J}_Y$.

Historically, Raki\'c first proved that any Osserman $R$ is Jacobi-dual (see \cite{Ra2}), 
while later some results related to the converse statement appeared (see \cite{A4} and \cite{BVM}).
We shall use the opportunity to show a new elementary and elegant proof that every Jacobi-dual $R$ is at least Einstein.

\begin{lema}
Any Jacobi-dual algebraic curvature tensor is Einstein.
\end{lema}
\begin{proof}
For an arbitrary unit $E_1\in\mc{V}$ there exists an orthonormal eigenbasis $(E_1,\dots,E_n)$ in $\mc{V}$ related to $\mc{J}_{E_1}$,
so there are $\lambda_1,\dots,\lambda_n\in\mb{R}$ that depend on $E_1$, such that $\mc{J}_{E_1}E_i=\lambda_i E_i$ holds for $1\leq i\leq n$.
If $R$ is Jacobi-dual then $\mc{J}_{E_i}E_1=\lambda_i E_1$ and considering the sharp of the Ricci tensor as a linear operator we have
$\Ric^{\sharp}=\mc{J}_{E_1}+\dots+\mc{J}_{E_n}$, and therefore $\Ric^{\sharp}(E_1)=(\lambda_1+\dots+\lambda_n)E_1$.
Hence, any unit $E_1\in\mc{V}$ is an eigenvector of $\Ric^{\sharp}$, so all eigenvalues of $\Ric^{\sharp}$ are the same, which gives
that $\lambda_1+\dots+\lambda_n$ is constant, and proves that $R$ is Einstein.
\end{proof}

Moreover, it turns out that the Jacobi-duality is a characterization of Osserman tensors, 
which was proved (using the perturbation theory) by Nikolayevsky and Raki\'c in 2013 \cite{NR1,NR2}.

\begin{result}\label{res1}
An algebraic curvature tensor is Osserman if and only if it is Jacobi-dual.
\end{result}

\section{Jacobi-orthogonal tensors}

In this work, we introduce a new concept of Jacobi-orthogonality. 
We say that an algebraic curvature tensor is Jacobi-orthogonal if $\mc{J}_XY\perp\mc{J}_YX$ holds for all $X\perp Y$.

\begin{theorem}\label{orto1}
Any Jacobi-orthogonal algebraic curvature tensor is Jacobi-dual.
\end{theorem}
\begin{proof}
Let $X,Y\in\mc{V}$ be mutually orthogonal unit vectors such that $\mc{J}_XY=\lambda Y$ holds for some $\lambda\in\mb{R}$.
Since $g(\mc{J}_YX,X)=g(\mc{J}_XY,Y)=\lambda$, we have $\mc{J}_YX=\lambda X+Z$ for some $Z\in\Span\{X,Y\}^{\perp}$.
The Jacobi-orthogonality immediately gives $\mc{J}_XY \perp \mc{J}_YX$, while
$X+Y\perp X-Y$ implies $\mc{J}_{X+Y}(X-Y) \perp \mc{J}_{X-Y}(X+Y)$.
However,
\begin{equation*}
\begin{aligned}
\mc{J}_{X+Y}(X-Y) &=(2\mc{J}_X+2\mc{J}_Y-\mc{J}_{X-Y})(X-Y)=2(\mc{J}_YX-\mc{J}_XY)\\
\mc{J}_{X-Y}(X+Y) &=(2\mc{J}_X+2\mc{J}_Y-\mc{J}_{X+Y})(X+Y)=2(\mc{J}_YX+\mc{J}_XY), 
\end{aligned}
\end{equation*}
which yields $\ve_{\mc{J}_XY}=\ve_{\mc{J}_YX}$.
Hence, $\lambda^2=\lambda^2+\ve_Z$, so we obtain $Z=0$ and consequently $\mc{J}_YX=\lambda X$, which proves the Jacobi-duality.
\end{proof}

According to Deep Result \ref{res1} any Jacobi-dual $R$ is Osserman, which yields the following consequence.

\begin{coro}\label{orto1c}
Any Jacobi-orthogonal algebraic curvature tensor is Osserman.
\end{coro}

Conversely, we want to show that any Osserman $R$ is Jacobi-orthogonal.
However, we manage to do this only in known cases.

\begin{theorem}\label{orto2}
Any Clifford algebraic curvature tensor is Jacobi-orthogonal.
\end{theorem}
\begin{proof}
Suppose $R$ is Clifford of form \eqref{clifford}, and calculate the Jacobi operator, 
\begin{equation*}
\mc{J}_XY=\mu_0(\ve_XY-g(Y,X)X)-3\sum_{i=1}^m \mu_ig(Y,J_iX)J_iX.
\end{equation*}
For $X\perp Y$, we have
\begin{equation*}
\begin{aligned}
g(\mc{J}_XY,\mc{J}_YX)&=g\left(\mu_0\ve_X Y-3\sum_i \mu_i g(Y,J_iX)J_iX,\quad \mu_0\ve_Y X-3\sum_j \mu_j g(X,J_jY)J_jY\right)\\
&=9\sum_{i,j}\mu_i\mu_j g(Y,J_iX)g(X,J_jY)g(J_iX,J_jY)\\
&=9\sum_{i,j}\mu_i\mu_j g(Y,J_iX)g(Y,J_jX)g(Y, J_jJ_iX)\\
&=\frac{9}{2}\sum_{i,j}\mu_i\mu_j g(Y,J_iX)g(Y,J_jX)g(Y, (J_iJ_j+J_jJ_i)X)\\
&=-9\sum_{i}\mu_i^2 g(Y,J_iX)^2g(Y,X)=0,
\end{aligned}
\end{equation*}
which proves that $R$ is Jacobi-orthogonal.
\end{proof}

\begin{theorem}\label{orto3}
Any two-root Osserman algebraic curvature tensor is Jacobi-orthogonal.
\end{theorem}
\begin{proof}
Suppose $R$ is two-root Osserman, so there exist two distinct eigenvalues $\lambda_1,\lambda_2\in\mb{R}$ such that for any nonzero $Y\in\mc{V}$
we have $\mc{V}=\Span\{Y\}\oplus \mc{V}_1(Y)\oplus \mc{V}_2(Y)$, where $\mc{V}_i(Y)=\Ker(\widetilde{\mc{J}}_Y-\ve_Y\lambda_i\id)$ for $i=1,2$.
Let us decompose $X\perp Y$ as $X=X_1+X_2$ with $X_1\in\mc{V}_1(Y)$ and $X_2\in\mc{V}_2(Y)$.
Since $\mc{J}_YX=\ve_Y(\lambda_1 X_1+\lambda_2 X_2)$, we have
\begin{equation*}
\begin{aligned}
g(\mc{J}_XY,\mc{J}_YX)&=\ve_Y R(Y,X_1+X_2,X_1+X_2,\lambda_1 X_1+\lambda_2 X_2)\\
&=\ve_Y(\lambda_2-\lambda_1) R(Y,X_1+X_2,X_1,X_2)=\ve_Y(\lambda_2-\lambda_1)(g(\mc{J}_{X_1}Y,X_2) -g(\mc{J}_{X_2}Y,X_1)).
\end{aligned}
\end{equation*}
Any Osserman $R$ is Jacobi-dual, which gives $g(\mc{J}_{X_1}Y,X_2)=g(\ve_{X_1}\lambda_1Y,X_2)=0$ and $g(\mc{J}_{X_2}Y,X_1)=0$,
so we obtain $\mc{J}_XY \perp \mc{J}_YX$, which proves that $R$ is Jacobi-orthogonal.
\end{proof}

According to Deep Result \ref{res2} any Osserman $R$ of dimension $n\neq 16$ 
is Clifford and therefore by Theorem \ref{orto2} it is Jacobi-orthogonal.
In dimension $16$ there exist Osserman curvature tensors that are not Clifford, 
which is the case with the model spaces $\mb{O}\mf{P}^2$ and $\mb{O}\mf{H}^2$.
However, all model spaces are one-root or two-root, and by Theorem \ref{orto3} they are Jacobi-orthogonal.

\begin{coro}\label{orto23c}
All known Osserman algebraic curvature tensors are Jacobi-orthogonal.
\end{coro}

We can conclude that the Jacobi-orthogonality is a very nice feature that would characterize Osserman tensors
in case that the Osserman conjecture is correct. The Jacobi-orthogonal condition brings us some interesting equations.

Let $R$ be an Osserman algebraic curvature tensor, and consequently $R$ is Jacobi-dual.
Let us fix an arbitrary unit $X\in\mc{V}$ and consider mutually orthogonal eigenvectors $A,B,C$ of ${\mc{J}}_X$ orthogonal to $X$.
There exist $\lambda_A,\lambda_B,\lambda_C\in\mb{R}$ such that $\mc{J}_XA=\lambda_AA$, $\mc{J}_XB=\lambda_BB$, and $\mc{J}_XC=\lambda_CC$, 
which yields $\mc{J}_AX=\ve_A\lambda_AX$, $\mc{J}_BX=\ve_B\lambda_BX$, and $\mc{J}_CX=\ve_C\lambda_CX$.
Hence $\mc{J}_X(A+B+C)=\lambda_AA+\lambda_BB+\lambda_CC$ and
\begin{equation*}
\begin{aligned}
\mc{J}_{A+B+C}X=& (\ve_A\lambda_A+\ve_B\lambda_B+\ve_C\lambda_C)X+ \mc{R}(X,A)B+\mc{R}(X,B)A\\
&+\mc{R}(X,A)C+\mc{R}(X,C)A+\mc{R}(X,B)C+\mc{R}(X,C)B,
\end{aligned}
\end{equation*}
while (since $R(X,A,B,A)=0$ and $R(X,B,A,B)=0$) we have 
\begin{equation*}
\mc{R}(X,A)B+\mc{R}(X,B)A=\frac{R(X,A,B,C)+R(X,B,A,C)}{\ve_C}C+D
\end{equation*}
where $D\in\Span\{X,A,B,C\}^{\perp}$, and therefore 
\begin{equation*}
\begin{aligned}
g(\mc{J}_X(A+B+C),\mc{J}_{A+B+C}X)=&\lambda_A (R(X,B,C,A)+R(X,C,B,A))+\lambda_B (R(X,A,C,B)+R(X,C,A,B))\\
&+\lambda_C (R(X,A,B,C)+R(X,B,A,C)). 
\end{aligned}
\end{equation*}
If we suppose that $R$ is Jacobi-orthogonal, then we obtain 
\begin{equation*}
R(X,A,B,C)(\lambda_C-\lambda_B) +R(X,B,A,C)(\lambda_C-\lambda_A) +R(X,C,A,B)(\lambda_B-\lambda_A)=0,
\end{equation*}
while the first Bianchi identity yields
\begin{equation*}
R(X,A,B,C)(\lambda_C-2\lambda_B+\lambda_A) +R(X,B,A,C)(\lambda_C+\lambda_B-2\lambda_A)=0.
\end{equation*}
The last formula characterizes the Jacobi-orthogonality condition for Osserman tensors, 
but unfortunately we do not see how to prove that any $k$-root Osserman $R$ is Jacobi-orthogonal for $k>2$.

\section*{Acknowledgements}

The authors was partially supported by the Ministry of Education, Science and Technological
Developments of the Republic of Serbia: grant number 451-03-68/2022-14/200104.

\end{document}